\documentclass[12pt]{article}
\usepackage[centertags]{amsmath}
\usepackage{amsfonts}
\usepackage{amssymb}
\usepackage{amsthm}
\usepackage{newlfont}
\usepackage[showonlyrefs]{mathtools}
\usepackage{comment}
\mathtoolsset{showonlyrefs}
\def\re{\operatorname{Re}}
\def\im{\operatorname{Im}}
\def\deg{\operatorname{deg}}
\def\cv{\operatorname{cv}}
\def\fix{\operatorname{Fix}}
\def\ann{\operatorname{ann}}

\def\Z{\mathbb{Z}}
\def\R{\mathbb{R}}
\def\C{\mathbb{C}}
\def\B{\mathcal{B}}
\newtheorem{theorem}{Theorem}[section]
\newtheorem{lemma}{Lemma}[section]
\newtheorem{corollary}{Corollary}[section]
\newtheorem{proposition}{Proposition}[section]

\theoremstyle{remark}
\newtheorem{remark-no}{Remark}[section]
\newtheorem*{remark}{Remark}
\newtheorem*{ack}{Acknowledgment}
\newtheorem{problem}{Problem}
\numberwithin{equation}{section}

\begin{document}

\title{On the multipliers of repelling periodic points of entire functions}

\author{Walter Bergweiler
and Dan Liu\thanks{Supported by the NNSF of China (Grant No.\ 11701188, 11371149).}}
\date{}
\maketitle

\begin{abstract}
We give a lower bound for the multipliers of repelling periodic points of entire functions.
The bound is deduced from a bound for the multipliers of fixed points of composite 
entire functions.

\medskip


\medskip

\end{abstract}

\section{Introduction and results}
Let $f$ be an entire function.
The iterates $f^n$ are defined by $f^0(z)=z$ and $f^n(z)=f(f^{n-1}(z))$ for $n \geq 1$.
A point $\xi\in\C$ is called a \emph{periodic point} of~$f$ if $f^n(\xi)=\xi$ for some $n\geq 1$.
The smallest $n$ with this property is the \emph{period} of~$\xi$.
A periodic point of period $1$ is called a \emph{fixed point}.
Let $\xi$ be a periodic point of period $n$ of $f$. Then $(f^n)'(\xi)$ is called
the \emph{multiplier} of $\xi$ and the point $\xi$ is called \emph{repelling} if
$|(f^n)'(\xi)| > 1$.

The periodic points play an important role in complex dynamics; for an introduction
to complex dynamics we refer to~\cite{Milnor1999,Steinmetz} for rational 
and~\cite{Bergweiler1993,Schleicher2010} for entire functions.
For example, the Julia set of a rational or entire function, which is defined as the
set where the iterates fail to be normal, is the closure of the set of repelling periodic points.
For rational functions this was proved by both Fatou and Julia (see~\cite[\S 14]{Milnor1999}
for an exposition of both proofs); for entire functions this result is
due to Baker~\cite{Baker1968}.
Before Baker's result it was not even known whether a transcendental entire function
must have repelling periodic points at all.
It turns out~\cite[Theorem~1]{Bergweiler1991} that a transcendental entire function $f$ actually has
repelling periodic points of period $n$ for every $n\geq 2$. In fact, for $n\geq 2$
there exists a sequence $(\xi_k)$ of repelling periodic points of period $n$ 
such that $|(f^n)'(\xi_k)|\to\infty$ as $k\to\infty$; see~\cite[Theorem~ 1.2]{Bergweiler2005}.
In this paper we will give a lower bound for $|(f^n)'(\xi_k)|$ in terms of 
the maximum modulus
\begin{equation}  \label{1a1}
M(r,f)=\max_{|z|=r}|f(z)|
\end{equation}
of~$f$, under a suitable additional hypothesis.

The Eremenko-Lyubich class $\B$ consists of all transcendental entire functions $f$ 
for which the set of critical and asymptotic values is bounded.
It was introduced in \cite{Eremenko1992} and  plays an important role in transcendental
dynamics; see~\cite[\S 3]{Schleicher2010}. For this class a lower bound for the multipliers of fixed points was 
given by Langley and Zheng~\cite[Theorem~2]{Langley}
who proved that there is a positive constant $c$ such that if $f\in\B$ 
and $0< \alpha <1$, then there exists a sequence $(\xi_k)$ of
fixed points of $f$ tending to $\infty$ such that 
\begin{equation}  \label{1a}
|f'(\xi_k)| > c\log M(\alpha |\xi_k|, f).
\end{equation}
Of course, this result does not hold outside the class $\B$, since an entire 
function need not have any fixed points at all.

Noting that if $f\in\B$ and $n\geq 1$, then $f^n\in\B$,
we see that there is also a sequence $(\xi_k)$ of
fixed points of $f^n$ such that 
\begin{equation}  \label{1b}
|(f^n)'(\xi_k)| > c\log M(\alpha |\xi_k|, f^n).
\end{equation}
It is reasonable to conjecture that for $n\geq 2$ a result of this type also holds outside
the class~$\B$.
We shall show that this is the case under an additional hypothesis
involving the minimum modulus 
\begin{equation}  \label{1a2}
m(r,f)=\min_{|z|=r}|f(z)| .
\end{equation}
More precisely, we shall require that there exist 
positive constants $a$ and $b$ with $b>1$ such that, for large~$r$,
\begin{equation}  \label{1a3}
\text{there exists} \ \rho  \in (r,br)
\ \text{such that}\  m(\rho,f) \leq a.
\end{equation}
This condition occurs in work of Rippon and Stallard on slow escaping points of 
entire functions~\cite[Theorem~2]{Rippon2011}. For example, it is clearly satisfied
if $f$ is bounded on a curve tending to~$\infty$. Functions in $\B$ satisfy
the latter condition.

We note that it follows from Harnack's inequality
that the existence of positive constants $a$ and $b$ satisfying~\eqref{1a3} is equivalent to 
the existence of, for any given~$\varepsilon>0$, a constant $c>1$ such that, for large~$r$,
\begin{equation}  \label{1a4}
\text{there exists} \ \rho  \in (r,cr)
\ \text{such that}\  \log m(\rho,f) \leq (1-\varepsilon)\log M(\rho,f);
\end{equation}
see the remark after Lemma~\ref{lemma5} below.
\begin{theorem}\label{thm1}
Let $f$ be a transcendental entire function satisfying~\eqref{1a3} and $n\geq 2$.
Then there exists a positive constant $c$ 
and a sequence $(\xi_k)$ of periodic  points of period $n$
tending to $\infty$ as $k\to\infty$ such that
\begin{equation}  \label{1c}
|(f^n)'(\xi_k)| > c\log M((1-o(1))|\xi_k|, f^n).
\end{equation}
\end{theorem}
We note that the bound given in~\eqref{1c} is of the right order of magnitude. Indeed,
if $f(z)=e^z$ and $f^n(\xi)=\xi$, then
\begin{equation}  \label{1c1}
(f^n)'(\xi)=\prod_{j=1}^n f^j(\xi)=\prod_{j=0}^{n-1} f^j(\xi)
\end{equation}
and hence
\begin{equation}  \label{1c2}
|(f^n)'(\xi)|\leq \prod_{j=0}^{n-1} f^j(|\xi|)=f^{n-1}((1+o(1))|\xi|)=
\log M((1+o(1))|\xi|, f^n)
\end{equation}
as $|\xi|\to\infty$.

Theorem~\ref{thm1} follows from a result concerning composite entire functions.
It was shown in~\cite[Theorem~3]{Bergweiler1991} that if $f$ and $g$ are transcendental entire
functions, then $f\circ g$ has infinitely many repelling fixed points. 
Here we prove the following result.
\begin{theorem}\label{thm2}
Let $f$ and $g$ be transcendental entire functions, with $f$
satisfying condition~\eqref{1a3}.
Then there exists positive constants $\alpha$ and $\beta$
such that $f\circ g$ has a sequence $(\xi_k)$ of fixed points
tending to $\infty$ and satisfying
\begin{equation}  \label{1d}
|(f\circ g)'(\xi_k)|\geq \alpha \log M(\beta M(|\xi_k|,g),f)
\end{equation}
and
\begin{equation}  \label{1e}
|g(\xi_k)| \geq (1-o(1))M(|\xi_k|, g).
\end{equation}
\end{theorem}
In our proofs we will use in part the ideas developed in~\cite{Bergweiler1991}, with the main tools coming
from Wiman-Valiron theory and Ahlfors' theory of covering surfaces.
In addition, we use an estimate of the spherical derivative based on
a method from~\cite[Theorem~5.2]{Eremenko2007}.
\begin{ack}
We thank the referees for many helpful suggestions.
\end{ack}

\section{Preliminary lemmas}
For $a\in\C$ and $r>0$, let $D(a,r)=\{z\in\C\colon |z-a|<r\}$ be the disk of radius 
$r$ centered at~$a$ and, for $R>r>0$, let $\ann(r,R)=\{z\in\C\colon r<|z|<R\}$
be the annulus centered at $0$ with radii $r$ and~$R$.

We begin by summarizing some results of Wiman-Valiron theory.
Let 
\[
g(z)=\sum\limits_{k=0}^{\infty}a_kz^k
\]
be a transcendental entire function.
Then 
\[
\mu(r,g):=\max_{k\geq 0}|a_k|r^k
\quad\text{and}\quad
\nu(r,g):=\max\{k\colon \mu(r,g)=|a_k|r^k\}
\]
are called the \emph{maximum term} and the \emph{central index}, respectively.
Note that $\mu(r, g)\leq M(r, g)$ by Cauchy's inequality.

Wiman-Valiron theory describes the behavior of an entire function $g$ near points of
maximum modulus; that is, near points $z_0$ for which $|g(z_0)|=M(|z_0|,g)$.
Essentially, it says that near such points, and
for $r=|z_0|$ outside an exceptional set $F$ of finite logarithmic measure, the  function
behaves like a polynomial of degree $\nu(r,g)$.
Here a subset $F$ of $[1,\infty)$ is said to have finite logarithmic measure 
if 
\[
\int_F \frac{dr}{r}<\infty.
\]
One consequence of the theory is that for each $C>0$ there exists $K>0$ such that 
$g(D(z_0,Kr/\nu(r,g)))\supset \ann(|g(z_0)|/C,C|g(z_0)|)$, if $z_0$ is a point of
maximum modulus and $|z_0|=r\notin F$ is sufficiently large.

The following lemma collects some of the main results of 
Wiman-Valiron theory; see~\cite{Hayman1974,Valiron1923}.
An alternative approach to results of this type is given in~\cite{Bergweiler2008},
with the central index $\nu(r,g)$ replaced by $d\log M(r,f)/d\log r$.

\begin{lemma}\label{lemma1}
Let $g$ be a transcendental entire function and let $K$ be a positive constant.
Then there exists a set $F$ of finite logarithmic measure such that 
if $|z_0|=r\notin F$, $|g(z_0)|= M(r,g)$  and $|\tau| \leq K/\nu (r,g)$, then
\begin{equation}  \label{2a}
g(z_0 e^{\tau})\sim g(z_0)e^{\nu (r,g) \tau},
\end{equation}  
\begin{equation}  \label{2a1}
|g(z_0 e^{\tau})|
\sim
M(|z_0 e^{\tau}|,g)
\end{equation}  
and
\begin{equation}  \label{2b}
g'(z_0e^{\tau}) \sim \frac{\nu (r,g)}{z_0 e^{\tau}}g(z_0 e^{\tau})       
\end{equation}  
as $r\to\infty$.
Moreover,
\begin{equation}  \label{2d}
\log \mu(r,g) \sim \log M(r,g)   
\end{equation}  
as $r\to\infty$, $r\notin F$, and, given $\varepsilon>0$,
\begin{equation}  \label{2c}
\nu(r,g) \leq (\log \mu(r,g))^{1+\varepsilon}    
\end{equation}  
for large $r\notin F$.
\end{lemma}
The following two lemmas are consequences of Lemma~\ref{lemma1} and Rouch\'{e}'s Theorem;
see \cite[Lemma~2]{Bergweiler1990a} and \cite[Lemma~3]{Bergweiler1990} for the proofs.
In both lemmas, $F$ is the exceptional set arising from Lemma~\ref{lemma1}.

\begin{lemma}\label{lemma2}
Let $g$ be a transcendental entire function,
$0<\varepsilon<2\pi$ and $K>0$.
Suppose that $|z_0|=r$, $|g(z_0)|=M(r,g)$ and $\sigma\in\C$ with
$|\sigma|\leq K$.
Then, if $r\notin F$ is large enough, there exists a unique $\tau\in\C$ such that
$|\nu (r,g)\tau- \sigma| < \varepsilon$ and $g(z_0 e^{\tau})=g(z_0)e^{\sigma}$. 
\end{lemma}

\begin{lemma}\label{lemma3}
Let $g$ be a transcendental entire function,
$K>0$ and $j\in\Z$.
Suppose that $|z_0|=r$ and $|g(z_0)|=M(r,g)$.
Then, if $r\notin F$ is large enough, there exists a 
holomorphic function $\tau_j\colon D(z_0,Kr/\nu(r, g))\to\C$
which satisfies
\begin{equation}  \label{2f}
g(ze^{\tau_j(z)}) = g(z),
\end{equation}  
and, as $r\to\infty$,
\begin{equation}  \label{2e}
|\tau_j(z)\nu(r, g) - 2\pi ij|\rightarrow 0
\end{equation}  
and
\begin{equation}  \label{2g}
\frac{d}{dz}(ze^{\tau_j(z)})\to 1 .
\end{equation}  
\end{lemma}

The following result is known as Harnack's inequality; see \cite[Theorem~1.3.1]{Ransford1995}.
\begin{lemma}\label{lemma4}
Let $v\colon  D(a,r)\rightarrow \R$ be a positive harmonic function, $0< \rho < 1$ and 
$z \in D(a,r)$ with $|z-a|\leq \rho r$. Then
\begin{equation}  \label{2h}
\frac{1-\rho}{1+\rho} \leq  \frac{v(z)}{v(a)} \leq  \frac{1+\rho}{1-\rho}.
\end{equation}  
\end{lemma}
A simple consequence of Harnack's inequality is the following lemma.
Similar results appear in \cite[Theorem~3]{Bergweiler2013}, \cite[Lemma~ 2]{Hinkkanen1994} 
and \cite[Lemma~5, part~(a)]{Rippon2011}.
\begin{lemma}\label{lemma5}
Let $R>0$ and $\varepsilon>0$. Let 
$u\colon \ann(e^{-2\pi/\varepsilon}R,e^{2\pi/\varepsilon}R)\to \R$
be a positive harmonic function.
Then
\begin{equation}  \label{2i}
\min_{|z|=R}u(z)> (1-\varepsilon)\max_{|z|=R}u(z).
\end{equation}  
\end{lemma}
\begin{proof}  
Choose $z_1$ and $z_2$ with 
$|z_1|=|z_2|=R$ such that
\begin{equation}  \label{2j}
u(z_1)=\max_{|z|=R}u(z)
\quad\text{and}\quad
u(z_2)=\min_{|z|=R}u(z).
\end{equation}  
Consider the strip
\begin{equation}  \label{2j1}
S= \left\{\zeta\in\C \colon \log R -\frac{2\pi}{\varepsilon}
<\re \zeta<\log R+\frac{2\pi}{\varepsilon} \right\}
\end{equation}  
and define
\begin{equation}  \label{2j2}
v\colon S\to\R,
\quad v(\zeta)=u(e^\zeta).
\end{equation}  
There exist $\zeta_1$ and $\zeta_2$ with
$\re \zeta_1=\re \zeta_2=\log R$ 
such that 
$z_1=\exp\zeta_1$, $z_2=\exp\zeta_2$
and 
$|\zeta_1-\zeta_2|\leq \pi$.
It now follows from Lemma~\ref{lemma4} with $a=\zeta_1$, $z=\zeta_2$, $r=2\pi/\varepsilon$
and $\rho=\varepsilon/2$ that
\begin{equation}  \label{2j3}
\frac{u(z_2)}{u(z_1)} =
\frac{v(\zeta_2)}{v(\zeta_1)} 
\geq \frac{1-\rho}{1+\rho}>1-2\rho
=1-\varepsilon,
\end{equation}  
from which the conclusion follows by the definition of $z_1$ and $z_2$.
\end{proof}  
\begin{remark}\label{remark5}
Let $f$ be an entire function satisfying~\eqref{1a4};
that is, for large $r$ there exists $\rho\in (r,cr)$ such that
$\log m(\rho,f)\leq (1-\varepsilon)\log M(\rho,f)$.
Applying Lemma~\ref{lemma5} to $u=\log|f|$
we see that $u$ cannot be positive in
the annulus $\ann(e^{-2\pi/\varepsilon}\rho,e^{2\pi/\varepsilon}\rho)$. Thus there 
exists 
\[
s\in (e^{-2\pi/\varepsilon}\rho,e^{2\pi/\varepsilon}\rho)\subset 
(e^{-2\pi/\varepsilon}r,e^{2\pi/\varepsilon}cr)
\]
with $m(s,f)\leq 1$.
Hence~\eqref{1a3} holds with $b=e^{4\pi/\varepsilon}c$.
Clearly~\eqref{1a3} implies~\eqref{1a4} and so
this shows that the conditions~\eqref{1a3} and~\eqref{1a4} are indeed equivalent,
as noted in the introduction.
\end{remark}

Pommerenke~\cite[Theorem~4]{Pommerenke} showed that there exists a constant $C>0$ such that 
if $f$ is a transcendental entire function,
then there exists a sequence $(z_k)$ tending to $\infty$ such that
\begin{equation}  \label{2l1}
|f(z_k)|\leq 1 
\quad\text{and}\quad
|f'(z_k)| \geq 
C\frac{\log M(|z_k|,f)}{|z_k|}.
\end{equation}  
We shall need the additional information that given a sequence $(w_k)$ satisfying 
$|f(w_k)|\leq 1$ for all~$k$, the sequence $(z_k)$ can be chosen such that $|z_k|$ and $|w_k|$ are of 
the same order of magnitude. This could be 
deduced from Pommerenke's method, but for completeness
we include the following lemma whose proof is based on a different method which was 
also used in~\cite{Eremenko2007,Bergweiler2012,Bergweiler2016}.
\begin{lemma}\label{lemma6}
Let $R>0$ and $h\colon \ann(e^{-\pi}R,e^{\pi}R)\to \C$
be holomorphic. Suppose that
\begin{equation}  \label{2k}
m(R,h)\leq 1<M(R,h).
\end{equation}  
Then there exists $z_0\in \ann(e^{-\pi}R,e^{\pi}R)$ such that
\begin{equation}  \label{2l}
|h(z_0)|= 1 
\quad\text{and}\quad
|h'(z_0)| \geq \frac{1}{2\pi}\frac{\log M(e^{-\pi}|z_0|,h)}{|z_0|}.
\end{equation}  
\end{lemma}
\begin{proof}
We proceed similarly as in the proof of Lemma~\ref{lemma5} and choose
$z_1$ and $z_2$ with $|z_1|=|z_2|=R$ such that
$|h(z_1)|=M(R,h)$ and $|h(z_2)|=m(R,h)$.
With
$S= \left\{\zeta\in\C \colon \log R -\pi <\re \zeta<\log R+\pi \right\}$
we define
$f\colon S\to \C$, $f(\zeta)=h(e^\zeta)$.
There exists $\zeta_1$ and $\zeta_2$ with $\re \zeta_1=\re \zeta_2=\log R$
such that
$z_1=\exp\zeta_1$, $z_2=\exp\zeta_2$ and
$|\zeta_1-\zeta_2|\leq \pi$.

Now let $G=\{\zeta\in S \colon |f(\zeta)|>1\}$  and $r=\max\{t\colon D(\zeta_1,t)\subset G\}$.
Then $0<r\leq |\zeta_1-\zeta_2|\leq \pi$ and
there exists $\zeta_0\in \partial G\cap\partial D(\zeta_1,r)\cap S$ with
$|f(\zeta_0)|=1$.
For $0<s<1$ we put $\zeta_s=(1-s)\zeta_0+s\zeta_1$ so that
$|\zeta_s-\zeta_0|=s|\zeta_1-\zeta_0|=sr$.
By the definition of $r$, the function $v\colon D(\zeta_1,r)\to\R$, $v(\zeta)=\log|f(\zeta)|$,
is a positive harmonic function.
Lemma~\ref{lemma4} 
(that is, Harnack's inequality),
applied with $a=\zeta_1$, $z=\zeta_s$ and $\rho=|\zeta_s-\zeta_1|/r=1-s$,
now yields that
\begin{equation}  \label{2h1}
\frac{v(\zeta_s)}{v(\zeta_1)} \geq  \frac{1-\rho}{1+\rho}> \frac{1-\rho}{2}=\frac{s}2
\end{equation}  
and hence
\begin{equation}  \label{2h2}
\frac{v(\zeta_s)-v(\zeta_0)}{|\zeta_s-\zeta_0|} 
=\frac{v(\zeta_s)}{s r}
>  \frac{v(\zeta_1)}{2r}
\geq  \frac{v(\zeta_1)}{2\pi}
=\frac1{2\pi}\log |f(\zeta_1)| 
=\frac1{2\pi}\log M(R,h).
\end{equation}
Thus
\begin{equation}  \label{2h3}
|f'(\zeta_0)|=\frac{|f'(\zeta_0)|}{|f(\zeta_0)|}=|\nabla v(\zeta_0)|
\geq \lim_{s\to 0} \frac{v(\zeta_s)-v(\zeta_0)}{|\zeta_s-\zeta_0|}
\geq \frac1{2\pi}\log M(R,h).
\end{equation}
With $z_0=\exp\zeta_0$ we thus have $|h(z_0)|=|f(\zeta_0)|= 1$ and
\begin{equation}  \label{2o}
|h'(z_0)|=\frac{|f'(\zeta_0)|}{|z_0|}
\geq \frac{1}{2\pi}\frac{\log M(R,h)}{|z_0|},
\end{equation}
which yields the conclusion since $|z_0|\leq e^{\pi}R$.
\end{proof}

The following lemma follows easily from Rouch\'e's theorem and
Schwarz's lemma; see \cite[Lemma 2.3]{Bergweiler2005}.
\begin{lemma} \label{lemma8}
Let $0<\delta<\varepsilon/2$ and let
$U\subset D(a,\delta)$ be a simply-connected domain.
Let $f\colon U\to D(a,\varepsilon)$  be holomorphic and bijective.
Then $f$ has a fixed point $\xi$ in $U$ which satisfies
$|f'(\xi)|\geq\varepsilon/4\delta$.
\end{lemma}

To achieve a situation where Lemma~\ref{lemma8} can be applied,
we use Ahlfors' theory of covering surfaces;
 see~\cite{Ahlfors1935}, \cite[Chapter~5]{Hayman1964} or \cite[Chapter~XIII]{Nevanlinna1953} 
for an account of Ahlfors' theory.
A key result of Ahlfors' theory is the following; see also~\cite{Bergweiler1998} for
another proof of this result.
Here $f^\#:=|f'|/(1+|f|^2)$ denotes the spherical derivative.
\begin{lemma} \label{lemma9}
Let $D_1$, $D_2$ and $D_3$ be three Jordan domains in $\C$ with pairwise disjoint closures. 
Then there exists a positive constant $A$ depending only on these domains such that if $a\in\C$,
$\rho>0$ and $f\colon D(a,\rho)\to\C$ is a holomorphic function with the property that 
no subdomain of $D(a,\rho)$ is mapped conformally onto one of the domains $D_j$, 
then
$\rho f^{\#}(a)\leq A$.
\end{lemma}
We mention that the result also holds for meromorphic functions if we take five Jordan
domains $D_j$ instead of three domains. This result is known as the Ahlfors five islands
theorem.

\section{Proof of Theorems \ref{thm1} and \ref{thm2}}

\begin{proof}[Proof of Theorem \ref{thm2}]
The idea of the proof is as follows.
Choose $z_0$ with $|z_0|=r$ and $|g(z_0)|= M(r,g)$ as in Lemmas~\ref{lemma1}--\ref{lemma3}.
Let $F$ be the exceptional set arising in these lemmas and suppose that
$r\not\in F$ is sufficiently large.
We consider a small disk $D(u_0,\rho)$ near $z_0$.
Lemma~\ref{lemma3} says that $g$ and hence $f\circ g$ map 
$D(u_0,\rho)$ and its images under the maps $z\mapsto \phi_j(z):=ze^{\tau_j(z)}$ 
in the same way.
We will choose $D(u_0,\rho)$ so that the images $\phi_j(D(u_0,\rho))$
are contained in pairwise disjoint disks $\Delta_j$.
Lemma~\ref{lemma9} will then be used to obtain $j\in\{1,2,3\}$ and a subdomain $U$ of $D(u_0,\rho)$
which is mapped conformally onto $\Delta_j$ by $f\circ g$. 
Since $\phi_j(U)\subset \Delta_j$ and $f\circ g\colon \phi_j(U)\to \Delta_j$ is conformal,
we can then apply Lemma~\ref{lemma8} to obtain a repelling fixed point of $f\circ g$.
For this we will actually replace $\Delta_j$ by a subdisk centered at $u_j=\phi_j(u_0)$
and replace $\phi_j(U)$ by the corresponding subdomain~$W$.

In order to apply Lemma~\ref{lemma9} we have to choose a point $u_0$ where the spherical derivative
is large.
This will be done using Lemma~\ref{lemma6}.
However, since Lemma~\ref{lemma9} asks for fixed disks $D_j$ while the $\Delta_j$
vary with $r$, we have to consider suitable rescalings of $f$ and $f\circ g$ and apply 
Lemmas~\ref{lemma6} and~\ref{lemma9} to those.

We thus define 
\begin{equation}  \label{3a}
H(w)= \frac{\nu (r,g)}{z_0}(f(w)-z_0)
\end{equation} 
and $R=M(r,g)$.
By our hypothesis~\eqref{1a3} there exists 
$w_1\in \ann(R,bR)$
satisfying $|f(w_1)|\leq a$.
It follows that 
\begin{equation}  \label{3c}
|H(w_1)|\leq \frac{\nu(r, g)}{r}(a+r) < 2\nu(r,g)
\end{equation} 
for large~$r$.
From~\eqref{2d} and~\eqref{2c} we can easily deduce that
\begin{equation}  \label{3d}
\log \nu(r,g)=o(\log M(r,g))
\end{equation} 
as $r\to\infty$, $r\notin F$.
Noting that $|w_1|> R=M(r,g)$ 
we conclude from~\eqref{3c} and~\eqref{3d} that 
\begin{equation}  \label{3e}
\log |H(w_1)|= o(\log M(r,g))= o(\log|w_1|)
\end{equation} 
as $r\to\infty$, $r\notin F$.
We choose $v_1$ with $|v_1|= |w_1|$ and $|f(v_1)|= M(|w_1|, f)$.
Noting that $M(|w_1|,f)\geq |w_1|\geq M(r,g)\geq 2r$ for large $r$, we find that
\begin{equation}  \label{3f}
|H(v_1)|= \frac{\nu (r, g)}{r}|f(v_1) - z_0| 
\geq  \frac{\nu (r, g)}{r}(M(|w_1|,f)-r) \geq 
\frac{\nu(r,g)}{2r}M(|w_1|, f).
\end{equation} 
Together with~\eqref{3d} this implies that 
\begin{equation}  \label{3g}
\log |H(v_1)|\geq (1-o(1)) \log M(|w_1|,f)
\end{equation} 
as $r\to\infty$, $r\notin F$.
Since $f$ is transcendental, we have 
\begin{equation}  \label{3h}
\frac{\log M(t,f)}{\log t}\to\infty
\end{equation} 
as $t\to\infty$.
Thus~\eqref{3e} and~\eqref{3g} yield that $\log|H(v_1)|>2\log|H(w_1)|$ for 
large $r\notin F$.
Applying Lemma~\ref{lemma5} with $\varepsilon=1/2$ we deduce
that $\log|H|$ cannot be a positive harmonic function in the annulus 
$\ann(e^{-4\pi}|w_1| ,e^{4\pi}|w_1|)$.
Hence $m(s,H)\leq 1$ for some $s\in (e^{-4\pi}|w_1| ,e^{4\pi}|w_1|)
\subset (e^{-4\pi}R,e^{4\pi}bR)$.

Lemma~\ref{lemma6} implies that there exists 
\begin{equation}  \label{3j}
w_0\in \ann(e^{-\pi}s,e^{\pi}s)\subset \ann(e^{-5\pi}R,e^{5\pi}bR)
\end{equation} 
such that
\begin{equation}  \label{3k}
|H(w_0)|= 1 
\quad\text{and}\quad
|H'(w_0)| \geq \frac{1}{2\pi}\frac{\log M(e^{-\pi}|w_0|,H)}{|w_0|}.
\end{equation} 
Using~\eqref{3d} we easily see that 
\begin{equation}  \label{3l}
\log M(e^{-\pi}|w_0|,H)\sim  \log M(e^{-\pi}|w_0|,f)
\end{equation} 
so that~\eqref{3k} yields
\begin{equation}  \label{3m}
\frac{\nu(r,g)}{r}|f'(w_0)|
=|H'(w_0)| \geq (1-o(1))\frac{\log M(e^{-\pi}|w_0|,f)}{2\pi|w_0|}.
\end{equation} 

By~\eqref{3j} we have
$w_0 = e^{\sigma}g(z_0)$ for some $\sigma$ satisfying $|\re\sigma|\leq  5\pi+\log b$
and $|\im \sigma|\leq \pi$.
Lemma~\ref{lemma2} implies that for large $r\notin F$ there exists $\tau$ satisfying
\begin{equation}  \label{3n}
|\tau|\leq \frac{6\pi +\log b}{\nu(r,g)}
\end{equation} 
and $g(z_0 e^\tau)=w_0$.

Put $u_0=z_0 e^\tau$ so that $g(u_0)=w_0$.
Note that by~\eqref{3n} we have $\tau\to 0$ and thus
\begin{equation}  \label{3n1}
u_0\sim z_0\quad\text{and}\quad |u_0|\sim r
\end{equation} 
as $r\to\infty$.

We will apply Lemma~\ref{lemma9} to the function $G\colon D(u_0,\rho)\to \C$,
\begin{equation}  \label{3o}
G(z)= \frac{\nu(r, g)}{u_0}(f(g(z))- u_0),
\end{equation} 
and the domains  $D_j= D (2\pi ij, 2)$ for $j\in\{1,2,3\}$,
with a value of $\rho$ still to be determined.

In order to do so, we have to estimate 
\begin{equation}  \label{3p}
\begin{aligned}
G^{\#}(u_0)
&=\frac{|G'(u_0)|}{1+|G(u_0)|^2}
\\ &
=\frac{\nu(r,g)}{|u_0|}\frac{|f'(w_0)|\cdot |g'(u_0)|}{1+|G(u_0)|^2}
\\ &
=(1+o(1))\frac{\nu(r,g)}{r}|f'(w_0)|\cdot |g'(u_0)| \cdot \frac{1}{1+|G(u_0)|^2}.
\end{aligned}
\end{equation} 
The first factor on the right hand side has already been estimated in~\eqref{3m}.
The second one can be estimated by~\eqref{2b} 
which yields
\begin{equation}  \label{3q}
|g'(u_0)|\sim \left| \frac{\nu(r,g)}{u_0}g(u_0) \right|\sim
\frac{\nu(r,g)}{r}|g(u_0)|= \frac{\nu(r,g)}{r}|w_0|.
\end{equation} 
In order to estimate the third factor on the right hand side of~\eqref{3p} we note
that
\begin{equation}  \label{3r}
|u_0-z_0|= r|e^\tau-1|=(1+o(1))r|\tau|\leq (1+o(1))\frac{(6\pi+\log b)r}{\nu(r, g)}
\end{equation} 
as $r\to\infty$, $r\notin F$, by~\eqref{3n}.
Since
\begin{equation}  \label{3s}
G(z)= \frac{z_0}{u_0} H(g(z))- \frac{\nu(r, g)}{u_0}(u_0-z_0)
\end{equation} 
we deduce from~\eqref{3k}, \eqref{3n1} and~\eqref{3r}
that
\begin{equation}  \label{3t}
|G(u_0)|=\left|\frac{z_0}{u_0}H(w_0)- \frac{\nu(r, g)}{u_0}(u_0-z_0)\right| 
\leq 1+6\pi+\log b+o(1).
\end{equation} 
Combining~\eqref{3p} with \eqref{3m}, \eqref{3q} and \eqref{3t} we find that
\begin{equation}  \label{3u}
G^{\#}(u_0)\geq (1-o(1)) 
\frac{\log M(e^{-\pi}|w_0|,f)}{2\pi |w_0|}
\cdot 
\frac{\nu(r,g)|w_0|}{r} 
\cdot
\frac{1}{ 1+(1+6\pi+\log b)^2 }.
\end{equation} 
Choosing $\eta$ with $0<\eta<1/(2\pi(1+(1+6\pi+\log b)^2))$ and putting $\gamma=e^{-\pi}$ we thus have
\begin{equation}  \label{3v}
G^{\#}(u_0)>  \eta  
\frac{\nu(r, g)}{r}\log M(\gamma|w_0|, f)
\end{equation} 
for large $r\notin F$.

We now put
\begin{equation}  \label{3w}
\rho = \frac{Ar}{\eta\,  \nu(r, g) \log M(\gamma|w_0|, f)},
\end{equation} 
and deduce from Lemma~\ref{lemma9}
that there exist $j\in \{1, 2, 3\}$ and a domain $U$ contained in $D(u_0, \rho)$
such that $G$ maps $U$ conformally onto $D_j$.
It follows that $f\circ g$ maps $U$ conformally onto 
\begin{equation}  \label{3w1}
\Delta_j:=
D\!\left(u_0\left(1+\frac{2\pi ij}{\nu(r,g)}\right),\frac{2|u_0|}{\nu(r,g)}\right)
\end{equation} 

We choose $\tau_j$ according to Lemma~\ref{lemma3},
for some fixed large constant~$K$, and define 
$\phi=\phi_j$ as mentioned above; that is,
\begin{equation}  \label{3x}
\phi\colon D\!\left(u_0,\frac{Kr}{\nu(r,g)}\right)\to\C,\quad\phi(z)=ze^{\tau_j(z)}.
\end{equation} 
Put $u_j=\phi(u_0)$.
Then 
\begin{equation}  \label{3y}
u_j
=u_0e^{\tau_j(u_0)}
=u_0\exp\!\left( \frac{2\pi ij+o(1)}{\nu(r,g)}\right) 
=u_0\left(1+ \frac{2\pi ij+o(1)}{\nu(r,g)}\right)
\end{equation} 
by~\eqref{2e}.
Thus 
\begin{equation}  \label{3z}
\Delta_j
\supset D\!\left(u_j,\frac{r}{\nu(r,g)}\right)
\end{equation} 
for large~$r$.
It follows that $U$ contains a subdomain $V$ which is mapped 
conformally onto $D(u_j,r/\nu(r,g))$ by $f\circ g$.

By~\eqref{2f} we have $g(z)=g(\phi(z))$ and~\eqref{2g} implies that
$\phi$ is univalent.
Putting $W=\phi(V)$ we see that $f\circ g$ maps $W$ conformally onto $D(u_j,r/\nu(r,g))$.
Moreover, we deduce from~\eqref{2g} that
\begin{equation}  \label{4a}
W=\phi(V)\subset \phi(U)\subset \phi(D(u_0,\rho))\subset  D(u_j,2\rho) 
\end{equation} 
for large~$r$.
We can now apply Lemma~\ref{lemma8} with $\delta=2\rho$ and $\varepsilon=r/\nu(r,g)$.
Note that the hypothesis $\delta<\varepsilon/2$ is clearly satisfied for large $r$
since~\eqref{3w} implies that
\begin{equation}  \label{4b}
\frac{\varepsilon}{\delta} =\frac{\eta }{2A}\log M(\gamma|w_0|,f)\to\infty
\end{equation} 
as $r\to\infty$.
Lemma~\ref{lemma8} now yields that $W$ contains a fixed point $\xi$ of $f\circ g$ satisfying
\begin{equation}  \label{4c}
|(f\circ g)'(\xi)|\geq \frac{\varepsilon}{4\delta} =\frac{\eta }{8A}\log M(\gamma|w_0|,f),
\end{equation} 
provided $r$ is large enough.
It follows from \eqref{3n}, \eqref{3w}, \eqref{3y} and~\eqref{4a}
that
\begin{equation}  \label{4d}
\begin{aligned} 
\xi
&=u_j\exp\!\left( \frac{o(1)}{\nu(r,g)}\right)   \\
&=u_0\exp\!\left( \frac{2\pi ij+o(1)}{\nu(r,g)}\right)   \\
&=z_0\exp\!\left(\tau +\frac{2\pi ij+o(1)}{\nu(r,g)}\right).
\end{aligned} 
\end{equation} 
Thus we deduce from~\eqref{2a} 
and~\eqref{2a1} that
\begin{equation}  \label{4e}
M(|\xi|,g)\sim |g(\xi)|
\sim \left| e^{\nu(r,g)\tau+2\pi ij}g(z_0)\right|=e^{\nu(r,g)\re\tau}|w_0|
\end{equation} 
and thus, by~\eqref{3n},
\begin{equation}  \label{4e1}
M(|\xi|,g) \leq (1+o(1))e^{6\pi +\log b} |w_0| .
\end{equation} 
Choosing $\alpha=\eta /(8A)$ and $\beta$ with $0<\beta<e^{-6\pi -\log b}\gamma=e^{-7\pi -\log b}$ we 
can now deduce from~\eqref{4c} and~\eqref{4e1} that 
\begin{equation}  \label{4f}
|(f\circ g)'(\xi)|\geq \alpha \log M(\beta M(|\xi|,g),f),
\end{equation} 
provided $r$ is large enough.
Since $|\xi|$ can be chosen arbitrarily large by taking $r$ large, the conclusion follows.
\end{proof}
\begin{proof}[Proof of Theorem \ref{thm1}]
We apply Theorem~\ref{thm2} with $g=f^{n-1}$
and obtain a sequence $(\xi_k)$ of fixed points of $f\circ g=f^n$ satisfying~\eqref{1d}
and~\eqref{1e}; that is,
\begin{equation}  \label{1d1}
|(f^n)'(\xi_k)|\geq \alpha \log M(\beta M(|\xi_k|,f^{n-1}),f)
\end{equation}
and
\begin{equation}  \label{1e1}
|f^{n-1}(\xi_k)| \geq (1-o(1))M(|\xi_k|, f^{n-1}).
\end{equation}

It follows easily from~\eqref{1e1} that if $1\leq m\leq n-1$ and $k$ is sufficiently 
large, then $\xi_k$ is not a fixed point of $f^m$.
Thus $\xi_k$ is a periodic point of period $n$ for large~$k$.

Since $f$ and hence $f^{n-1}$ are transcendental, it follows that if $K>1$, then
\begin{equation}  \label{4g}
\frac{M(Kr,f^{n-1})}{M(r,f^{n-1})}\to\infty
\end{equation}
as $r\to\infty$. This implies that 
$\beta M(r,f^{n-1})=M((1-o(1))r,f^{n-1})$
as $r\to\infty$.
Together with the maximum principle this yields that 
\begin{equation}  \label{4i}
M(\beta M(r,f^{n-1}),f)= M(M((1-o(1))r,f^{n-1}),f)
\geq M((1-o(1))r,f^{n}).
\end{equation}
Combining this with~\eqref{1d1} we obtain the conclusion.
\end{proof}

\section{Remarks and open questions}
Let $f$ be a transcendental entire function that does not satisfy the condition~\eqref{1a3}
imposed in our results.
Using that this condition is equivalent to~\eqref{1a4} we can deduce that there exists
a sequence $(R_k)$ tending to $\infty$ and a sequence $(G_k)$ of Jordan domains
satisfying $0\in G_k\subset D(0,R_k)$ such that 
$f\colon G_k\to D(0,R_k)$ is a proper map for all~$k$. In the terminology of 
Douady and Hubbard~\cite{Douady1985}, the map $f\colon G_k\to D(0,R_k)$ is a 
polynomial-like map.
As indicated by the name, such functions behave like polynomials in some sense, and
this was the underlying idea in the proof in~\cite{Bergweiler1991}
that $f^n$ has infinitely many repelling periodic points of any period $n\geq 2$
in this case, and that $f\circ g$ has infinitely many repelling fixed points
for any transcendental entire function~$g$. 

While not phrased this way in~\cite{Bergweiler1991}, the essential tool from
complex dynamics that was used is the following proposition.
Here we denote by $\deg(p)$ the degree and by
$\cv(p)$ the number of critical values of a polynomial~$p$; 
that is, the cardinality of the set $p((p')^{-1}(0))$.
\begin{proposition}\label{prop1}
Let $p$ be a polynomial satisfying $\deg(p)> 2 \cv(p)$. 
Then $p$ has a repelling fixed point.
\end{proposition}
\begin{proof}[Sketch of proof]
Let $\xi$ be a fixed point of $p$ and let $\lambda=p'(\xi)$ be its multiplier.
If $\xi$ is attracting, that is, if $|\lambda|<1$, then -- 
by a standard result~\cite[\S 8]{Milnor1999} in complex dynamics --
the attracting basin $\{z\colon p^n(z)\to\xi\}$ contains a critical value of~$p$.
If $\lambda=1$, then $\xi$ is a zero of $z\mapsto p(z)-z$ of multiplicity $m\geq 2$.
Again by classical results~\cite[\S 10]{Milnor1999} in complex dynamics,
there are $m-1$ invariant domains where the iterates of $p$ tend to $\xi$, 
and each such domain contains a critical value of~$p$.

It follows from these considerations that, counting multiplicities,
there are at most $2\cv(p)$ fixed
points of $p$ for which the multiplier $\lambda$ satisfies $|\lambda|<1$ or $\lambda=1$.
Now let $N$ be the number of non-repelling fixed points.
By perturbing $p$ so that fixed points with multiplier $\lambda$ satisfying $|\lambda|=1$
but $\lambda\neq 1$ become attracting, one can also show that 
$N\leq 2\cv(p)$.
Since $p$ has $\deg(p)$ fixed points, the conclusion follows.
\end{proof}
We note that the perturbation needed to pass from $|\lambda|=1$
to $|\lambda|<1$ does not need the results of Shishikura~\cite{Shishikura1987} about
quasiconformal surgery or the result of Douady and Hubbard~\cite{Douady1985}
that polynomial-like maps are quasiconformally conjugate to polynomials. 
Instead, using interpolation one perturbs $p$ to a polynomial $q$ of higher degree.
But viewed as a polynomial-like map on a suitable domain, the polynomial $q$
has the same degree as $p$; see~\cite[p.~67]{Bergweiler1991} and~\cite[p.~54]{Douady1983}
for more details.
The argument shows that Proposition~\ref{prop1} extends to polynomial-like maps.

To see how Proposition~\ref{prop1} can be applied to prove that composite polynomials have
repelling fixed points we note that if $p$ and $q$ are polynomials,
then 
\begin{equation}  \label{5a}
\cv(p\circ q) \leq \cv(p)+ \deg(q)-1 \leq \deg(p)+\deg(q)-2
\end{equation}
while
\begin{equation}  \label{5b}
\deg(p\circ q) = \deg(p)\cdot \deg(q).
\end{equation}
Proposition~\ref{prop1} yields that $p\circ q$ has a repelling fixed point if 
\begin{equation}  \label{5c}
\deg(p)\cdot \deg(q)>  2(\deg(p)+\deg(q)-2),
\end{equation}
which is equivalent to 
\begin{equation}  \label{5d}
(\deg(p)-2)\cdot (\deg(q)-2) > 0. 
\end{equation}
We thus have the following corollary of Proposition~\ref{prop1}.
\begin{corollary}\label{cor1}
Let $p$ and $q$ be polynomials of degree at least~$3$.
Then $p\circ q$ has a repelling fixed point.
\end{corollary}

Proposition~\ref{prop1} yields -- under the hypotheses made -- 
the existence of a repelling fixed point, but it does not give any further information about
the multiplier of this fixed point.
It seems plausible that if $2\cv(p)$ is much smaller than~$\deg(p)$, then there should 
be a fixed point of large multiplier.

Denote by $\fix(p)$ the set of fixed points of a polynomial $p$ and let
\begin{equation}  \label{5e}
\Lambda(p)=\max\{|p'(\xi)|\colon \xi \in \fix(p)\}
\end{equation}
be the maximum of the moduli of the multipliers of the fixed points of~$p$.

\begin{problem}\label{prob1}
Let $p$ be a polynomial. Can one give a lower bound for $\Lambda(p)$ depending only
on $\deg(p)$ and $\cv(p)$ such that $\Lambda(p)>1$ if  $\deg(p)> 2 \cv(p)$?
\end{problem}
As mentioned, one would expect that
 $\Lambda(p)$ is large if $\cv(p)$ is small compared to~$\deg(p)$.
In the extremal case that $\cv(p)=1$ the polynomial $p$ is conjugate to the polynomial
$z\mapsto z^d$, with $d=\deg(p)$, and it follows that $\Lambda(p)=d$ in this case.
\begin{problem}\label{prob2}
Let $\varepsilon>0$. Does there exist $\delta>0$ such that if $p$ is a polynomial
satisfying $\cv(p)<\delta \deg(p)$, then $\Lambda(p)\geq (1-\varepsilon)d$?
\end{problem}
An answer to these questions could help to give a lower bound for the multipliers
of periodic points or fixed points of composite functions,
comparable to~\eqref{1c} or~\eqref{1d},
in the case that condition~\eqref{1a3} is not satisfied.

In the context of iteration rather than composition we also have
the following question~\cite[Conjecture~C]{Bergweiler2010}.
\begin{problem}\label{prob3}
Let $p$ be a polynomial of degree $d\geq 2$ and let  $n\geq 2$.
Do we have $\Lambda(p^n)\geq d^n$?
\end{problem}
The monomial $p(z)=z^d$ already considered shows that this would be best possible.
Eremenko and Levin~\cite[Theorem~3]{Eremenko}
have shown that if $p$ is a polynomial of degree $d\geq 2$ which is
not conjugate to the monomial $z\mapsto z^d$, then $\Lambda(p^n)> d^n$ 
for some $n\geq 2$.
The question is whether this holds for all $n\geq 2$.

Instead of considering the multipliers of all fixed points of $p^n$ one may 
also restrict to those of periodic points of period~$n$ and study the 
corresponding modification of $\Lambda(p^n)$. The polynomials that 
fail to have repelling periodic points of some period were classified
in~\cite{Bergweiler1991a,Chang2006}.

\noindent Mathematisches Seminar\\
Christian-Albrechts-Universit\"at zu Kiel\\
Ludewig-Meyn-Str.\ 4\\
24098 Kiel, Germany

\bigskip

\noindent Department of Applied Mathematics\\
South China Agricultural University\\
Guangzhou, 510642, P.~R.\ China

\end{document}